\documentclass[12pt,a4paper]{article}
\usepackage[active]{srcltx}
\usepackage{amsfonts}
\usepackage{amsmath}
\usepackage{amsbsy}
\usepackage{amsxtra}
\usepackage{latexsym}
\usepackage{amssymb}
\usepackage{url}
\usepackage{cite}
\usepackage{pstricks,pst-node}
\makeatletter
\g@addto@macro\thesection.
\makeatother

\newtheorem{theorem}{Theorem}

\newtheorem{definition}{Definition}

\DeclareMathOperator{\intr}{int}

\DeclareMathOperator{\bdr}{bdr}

\DeclareMathOperator{\gauge}{gauge}

\newcommand{\R}{\mathbb R}

\newenvironment{proof}{{\noindent\bf Proof.}}{\hfill$\Box$\\}

\begin{document}

\title{Geometry of the proper asymmetric norm}
\author{A. B. N\'emeth\\Faculty of Mathematics and Computer Science\\Babe\c s 
Bolyai University, Str. Kog\u alniceanu nr. 1-3\\RO-400084
Cluj-Napoca, Romania\\email: nemab@math.ubbcluj.ro}
\date{}
\maketitle

\begin{abstract} 

The convex geometric approach in the study of asymmetric norms can be useful
in their deeper investigation. The note illustrates this in
the case of so called proper asymmetric norms, notion revealed in
analytical context in \cite{NemethNemeth2020} and \cite{Nemeth2020}.   

\end{abstract}

\section{Introduction}

While searching for asymmetric vector norms \cite{NemethNemeth2020} and mutually
polar retractions on convex cones \cite{Nemeth2020}, retractions on one dimensional
cones play a special role. These retractions are intrinsically  related to some
special, so called \emph{proper} asymmetric norms. We will show in this note
that the analytic conditions on these 
asymmetric norms have simple interpretation in convex geometric setting.  

\begin{definition}\label{assnor}\cite{Cobzas2013}
Let $X$ be a real vector space.

The functional $p:\,X\to \R_+=[0,+\infty)$ is said an \emph{asymmetric norm}
if the following conditions hold:
\begin{enumerate} 
\item $p$ is positive homogeneous, i.e., $q(tx)=tq(x),\, \forall t\in \R_+,\,\,\forall x\in X$;
\item $p$ is subadditive, i.e., $p(x+y)\leq p(x)+p(y),\,\,\forall x,\,y \in X$;
\item If $p(x)=p(-x)=0$ then $x=0$. 
\end{enumerate}
\end{definition}

The theory of asymmetrically normed spaces concerns merely on
the investigation and  relation of the various $T_0$ topologies 
induced by asymmetric norms on $X$ and the related functional analytic problems. 
In our approach in \cite{NemethNemeth2020} and \cite{Nemeth2020}
the difference lies in the fact that we consider a priori a norm on $X$ and
that in the Definition \ref{assnor} we have supposed $p$ continuous.

\begin{definition}\label{pran}
Adapting the terminology from \cite{NemethNemeth2020},
an asymmetric norm $q:\,X\to \R_+$ is said \emph{proper}
if there exists an element $u\in X$ with $q(u)=1$, such that 
\begin{equation}\label{pran1}
q(x-q(x)u)=0\;\;\forall\;x\in X.
\end{equation}
\end{definition}

The note is organized as follows:

After the definition of cones and ordering in a short Section 2,
in Section 3 we give the complete geometric characterization 
of the proper asymmetric norm. In Section 4 we use the
result of the preceding section in the study of mutually polar
retractions on convex cones of a normed space. The obtained results 
generalize some results in \cite{NemethNemeth2020} and
\cite{Nemeth2020}. 
\section{Cones and orderings}

Let $X$ be a real vector space. The set $K\subset X$ is called a 
\emph{pointed convex cone} if (i) $x,\,y\in K\,\Rightarrow x+y\in K$,
(ii) $x\in K,\,t\in \R_+=[0,+\infty)\;\Rightarrow tx\in K$, and
(iii) $K\cap (-K)=\{0\}.$

The relation
$$ x\leq_Ky \;\Leftrightarrow\; y-x\in K$$
defines a reflexive, transitive and anti-symmetrical order relation on $K$.

In the convex geometry a cone in a normed space is called \emph{proper}
if it is convex, pointed, closed and possesses interior points.
(In \cite{Conradie2015} the proper cone means pointed cone.)


\section{The geometry of asymmetric norms}

We shall be in keeping with the terminology in \cite{Cobzas2013} and
\cite{Conradie2015}. 

Let $X$ be a real vector space. 

The set $A\subset X$ is  called \emph{absorbent} if
$$\forall \,x\in X\; \exists t\in \R^+=(0,+\infty) \;\textrm{such that}\;x\in tA.$$

The set $A$ is \emph{absorbent with respect to} $a\in A$, if $A-a$ is absorbent.

In the presence of a locally convex topology on $X$, every non-empty
convex set $C$ in $X$ is absorbent with respect to every its interior point.

If $A$ is  absorbent, then the functional $g:X\to \R_+$ defined with
$$  g(x)=\inf\{t\in \R^+:\,x\in tA\}$$ 
is called the \emph{gauge of} $A$ and is denoted $\gauge A.$

If $A-a$ with $a\in A$ is absorbent, then
$\gauge(A-a)$ is said the \emph{gauge of $A$ with respect to its point $a$}.

If $p:X\to \R_+$ is the asymmetric norm
in Definition \ref{assnor}, then we put
$$B'_p=\{x\in X:\,p(x)<1\},$$ 
$$B_p=\{x\in X:\,p(x)\leq 1\},$$
and
$$\Sigma_p=\{x\in X:\,p(x)=1\}.$$

Obviously, $B'_p$ and $B_p$ are both convex and absorbent and
$p$ is the gauge of both $B'_p$ and $B_p$.

The family
$$ \{tB'_p:\,t\in \R^+\}$$
form the neighborhood basis of a $T_0$ topology in $X$ which
we shall call the $p$-\emph{topology}.

The functional $p^s:X\to \R_+$ defined by
$$p^s(x)=\max \{p(x),p(-x)\}$$
is a norm on $X$ called the $p^s$-\emph{norm}.

Since $p\leq p^s$, the functional $p$ is $p^s$-continuous.

The set 
$$K_p=\{x\in X:\, p(x)=0\}$$
is a pointed convex cone.
Indeed $x,y\in K_p \Rightarrow x+y\in K_p$,
$x\in K_p,\;t\in \R_+\;\Rightarrow tx\in K_p$ and
 $x,\,-x\in K_p\,\Rightarrow x=0.$
$K_p$ is also $p^s$-closed since $p$ is $p^s$-continuous.

\begin{theorem}\label{main}
Suppose that $X$ is a real vector space and $q:X\to \R_+$
is an asymmetric norm. Then the following assertions are
equivalent:
\begin{enumerate}
\item [(i)] $q$ is proper, that is there exists $u\in X$ with
$q(u)=1$ such that
$$q(x-q(x)u)=0\;\forall \; x\in X.$$
\item [(ii)]
\begin{equation}\label{sig}
q(x-u)=0\;\forall \; x\in \Sigma_q.
\end{equation}
\item [(iii)] $\{x-q(x)u: \;x\in X\}=K_q$, where $K_q=\{x\in X:\,q(x)=0\}$.
\item [(iv)] There exists a norm on $X$ such that $q$ is
the gauge with respect to an interior point of
a proper cone.
\end{enumerate}
 \end{theorem}
\begin{proof}

The relations (i) $\;\Leftrightarrow\;$ (ii) $\;\Leftrightarrow \;$ (iii) are obvious.

(ii) $\;\Rightarrow\;$ (iv).

From (\ref{sig}) we have that $q$ is zero on the  set $\Sigma_q-u$. This set is nothing else as
the translation of the boundary of $B_q$ with $-u$. Since $B_q-u$ is a convex set
and the non-negative convex functional $q$ vanishes on the boundary of $B_q-u$, it
must be zero on this convex set. Hence
\begin{equation}\label{sub}
B_q-u\subset K_q.
\end{equation}

$B'_q-u$ is $q^s$ open and it is contained in the $q^s$-closed pointed cone $K_q$,
whereby this cone is $q^s$-proper.

Let $k\in K_q$ be arbitrary. Then $y=k+u\in B_q$  since $q(y)=q(k+u)\leq q(k)+q(u)=q(u)=1.$
Hence $k=y-u$, where $y\in B_q$. It follows that $K_q\subset B_q-u.$
Together with (\ref{sub}) we have
$$B_q-u\subset K_q\subset B_q-u.$$
Hence $B_q$ is the translation with $-u$ of the $q^s$-proper
cone $K_q$, and 
$$q=\gauge(K_q+u).$$

$-u$ is a $q^s$-interior point of $B_q-u=K_q.$
Thus $q$ is the gauge of the $q^s$-proper cone $K_q$ in a normed space with respect to
its interior point $-u$. This is nothing else as (iv).

(iv) $\;\Rightarrow\;$ (ii).

Suppose that $X$ is a normed vector space and $K\subset X$ is a proper cone.

Take $-u\in \intr K$ and put 
$$q=\gauge (K+u),$$
that is let $q$ be the gauge of $K$ with respect to its interior point $-u$.
Then $q$ is an asymmetric norm,
$$B_q= K+u,$$
and
$$K_q=K.$$

Then $\Sigma_q= \bdr(K+u)=\bdr K+u$, hence $u\in \Sigma_q$ and $q(u)=1$.

$\forall\, x\in \Sigma_q$, $x-u\in \bdr K \subset K,$ hence
$q(x-u)=0,$ which is nothing else as (\ref{sig}), that is, 
we have condition (ii) fulfilled for $q$.

\end{proof}

In conclusion {\bf an asymmetric norm is proper if and only if it is the gauge
of a proper cone with respect to its interior point.}

\section{Application to mutually polar retractions}

Let $X$ be a real normed space. The mapping $S:X\to X$ is called \emph{idempotent}
if $S^2=S$.

\begin{definition}\label{retract} 
The mapping $T:\,X\to X$ is called
\emph{retraction} if:
\begin{enumerate}
\item [(i)] It is a continuous idempotent mapping;
\item [(ii)] It is \emph{positive
homogeneous}, that is, $T(tx)=tTx$ for every $x\in X$ and every $t\in \R_+=[0,+\infty)$;
\item [(iii)] $T(X)$ is a non-empty, non-zero, closed pointed convex cone.
\item [(iv)] $Tx\in \bdr T(X)$ for any $x\in X\setminus T(X).$
\end{enumerate}
\end{definition} 

If $S:X\to X$ is a retract with $K=S(X)$, then it is called \emph{subadditive} if
$$S(x+y)\leq_K Sx+Sy.$$
We will call a retract subadditive, if it is subadditive with respect to the
ordering its range endows.

\begin{definition}\label{muture}
Let $X$ be a real normed space, $0$ its zero mapping and $I$ its identity mapping. 
The mappings $Q,\,R:\,X\to X$ are called
\emph{mutually polar retractions} if 
\begin{enumerate}
\item [(i)] $Q$ and $R$ are retractions, 
\item[(ii)] $Q+R=I$,
\item[(iii)] $QR=RQ=0$.
\end{enumerate}
\end{definition}

Next we use the notations and definitions
in the preceding section. 

\begin{theorem}\label{as1}
Let $Q,\,R:X\to X$ be a pair of mutually polar retractions with $Q(X)=M$ and
$R(X)=N$.
If $\dim N=1$ and $R$ is subadditive then
\begin{enumerate}
\item [(i)] The functional $q$ from the representation $Rx=q(x)u$ with $q(u)=1$, is a
proper asymmetric norm.
\item [(ii)] $M=K_q=\{x\in X:q(x)=0\}$ is a proper cone, $-u\in \intr M$.
\item [(iii)] $q=\gauge (M+u).$
\item [(iv)] $Q$ is subadditive.
\end{enumerate}
\end{theorem}

\begin{proof}

(i) Obviously, $R$ can be represented in the form $Rx=q(x)u$ with $u\in N$
and $q(u)=1$, where $q$ is subadditive and positively homogeneous since $R$ is so.

If $q(x)=q(-x)=0,$ then by $Q+R=I$ we have $x=Qx\in M$ $-x=Q(-x)\in M,$
hence $x=0$ since $M$ is pointed by definition.

Since $Q=I-R$ is idempotent, we have
$$ x-q(x)u = (I-R)x=(I-R)^2x =(I-R)(I-R)x= (I-R)x -R(I-R)x=$$
$$x-q(x)u-q(x-q(x)u)u,\;\forall \,x.$$
Hence
$$q(x-q(x)u)=0,\;\;\forall\; x\in X,$$
and thus $q$ is a proper asymmetric norm.

(ii) If $x\in K_q$ , then $x=Qx+q(x)u= Qx \in M.$ Thus $K_q\subset M$.
Since $RQ=0$ we have $R(Qx)=q(Qx)u=0\;\forall \,x\in X,$ that is, $Qx\in K_q,\;
\forall \; x\in X.$ Hence $M=Q(X)\subset K_q.$

From Theorem \ref{main}, $K_q=M$ is a proper cone and $-u\in \intr M$.

(iii) From the same Theorem, \ref{main} $q= \gauge (u+K_q) = \gauge (M+u).$

(iv) We have to see that
\begin{equation}\label{iv}
 Qx+Qy-Q(x+y) \in M,\;\;\forall \;x,\,y\in X.
\end{equation}
From $Q=I-R$ the left hand side of this relation rewrites as
$$ x-Rx+y-Ry-(x+y-R(x+y))=-Rx-Ry +R(x+y).$$
Now, $-Rx-Ry+R(x+y)\in -N$ as $R$ is subadditive. Since by (ii) 
$-N =\{-tu:\,t\in \R_+\}\subset M,$
we have (\ref{iv}) fulfilled.

\end{proof}

\begin{theorem}\label{as2}

If $M\subset X$ is a proper cone with $-u\in \intr M$ and 
$q=\gauge (M+u)$, then
\begin{enumerate}
\item [(i)]  $$Rx=q(x)u ,\; x\in X,\;\textrm{and}\; Q=I-R$$
are mutually polar retractions with $R(X)=\{tu:t\in \R_+\}$ and
$Q(X)=M.$
\item [(ii)] $Q$ and $R$ are subadditive.
\end{enumerate}
\end{theorem}

\begin{proof}

(i) From Theorem \ref{main}, $q$ is a proper asymmetric norm
with
$$K_q=M.$$
Hence
$$(I-R)^2x =(I-R)(I-R)x= (I-R)x -R(I-R)x=$$
$$x-q(x)u-q(x-q(x)u)u= x-q(x)u=(I-R)x ,\;\forall \,x\in X,$$
and it follows that $Q=I-R$ is idempotent.

$Qx=x-q(x)u=x$
if and only if $q(x)=0$ i.e., if $x\in M$. Thus $Q(X)=M.$

(ii) $R$ is subadditive and $-R(X)=-\{tu:t\in \R_+\}\subset M$
and repeating the argument in the proof of item (iv) of Theorem \ref{as1}, it
follows that $Q$ is subadditive too.

\end{proof}

In conclusion:  {\bf Two mutually polar retractions $Q,\,R :\,X\to X$ with
$\dim R(X)=1\;,$ $R(x)=q(x)u,\;q(u)=1$ are subadditive, if and only if 
$Q(X)$ is a proper cone, $-u\in \intr Q(X)$ and $q=\gauge(M+u)$.}
 

\end{document}